\date{27 August, 2016}
\documentclass[11pt]{amsart}
\usepackage[utf8]{inputenc}
\usepackage[T1]{fontenc}
\usepackage{amsmath}
\usepackage{amsfonts}
\usepackage{amssymb}
\usepackage{amsthm}
\usepackage[all,cmtip]{xy}
\usepackage{mathrsfs}
\usepackage{mathabx}
\usepackage{color}
\definecolor{cobalt}{RGB}{61,89,171}
\usepackage[colorlinks,citecolor=cobalt,linkcolor=cobalt,urlcolor=cobalt,pdfpagemode=UseNone,backref = page]{hyperref}

\def\bQ{\mathbb{Q}}
\def\bR{\mathbb{R}}
\newcommand{\R}{\mathbb{R}}

\newcommand{\cF}{\mathcal{F}}
\def\cH{\mathcal{H}}
\def\cL{\mathcal{L}}
\def\cM{\mathcal{M}}

\renewcommand{\a}{\alpha}

\newcommand{\ii}{\imath}

\theoremstyle{plain}
\newtheorem{theorem}{Theorem}[section]
\newtheorem{proposition}[theorem]{Proposition}
\newtheorem{lemma}[theorem]{Lemma}
\newtheorem{corollary}[theorem]{Corollary}

\theoremstyle{definition}
\newtheorem{definition}[theorem]{Definition}

\newtheorem{remark}[theorem]{Remark}

\newcommand{\thmref}[1]{Theorem~\ref{#1}}

\newcommand{\lemref}[1]{Lemma~\ref{#1}}
\newcommand{\corref}[1]{Corollary~\ref{#1}}

\begin{document}

\title{Parallel Forms, Co-K\"ahler Manifolds and their Models}

\author[G. Bazzoni]{Giovanni Bazzoni}
\address{Fachbereich Mathematik und Informatik\\
Philipps-Universit\"at Marburg\\
35032 Marburg, Germany}
\email{bazzoni@mathematik.uni-marburg.de}

\author[G. Lupton]{Gregory Lupton}
\address{Department of Mathematics\\
Cleveland State University\\
Cleveland OH \\
44115  USA}
\email{g.lupton@csuohio.edu}

\author[J. Oprea]{John Oprea}
\address{Department of Mathematics\\
Cleveland State University\\
Cleveland OH \\
44115  USA}
\email{j.oprea@csuohio.edu}

\date{\today}


\begin{abstract} We show how certain topological properties of co-K{\"a}hler manifolds
derive from those of the K\"ahler manifolds which construct them. In particular, we show that the existence of parallel
forms on a co-K\"ahler manifold reduces the computation of cohomology from the
de Rham complex to certain amenable sub-cdga's defined by geometrically natural
operators derived from the co-K\"ahler structure. This provides a simpler proof of the formality of the foliation minimal model 
in this context.
\end{abstract}

\thanks{This work was partially supported by grants from the Simons Foundation (\#209575 to
Gregory Lupton and \#244393 to John Oprea) and an INdAM (Istituto Nazionale di Alta Matematica)
fellowship (to Giovanni Bazzoni).}

\maketitle

\section{Introduction}\label{sec:intro}
Co-K\"ahler manifolds may be thought of as odd-dimensional versions of K\"ahler manifolds and various structure theorems explicitly display how the former are constructed
from the latter (see \cite{BO,Li}).

In this paper, we take the point of view that topological and geometric properties of co-K\"ahler manifolds are inherited from those of the K\"ahler manifolds 
that construct them. We call this the \emph{hereditary principle} and we shall see this in both topological and geometric contexts. See \cite{BLO} for further  
applications of this principle. First, let us recall some basic definitions (see \cite{Bl} for a detailed introduction).

\begin{definition}\label{def:cosymp}
An \textbf{almost contact metric structure} $(J,\xi,\eta,g)$ on a manifold $M^{2n+1}$ consists of a tensor $J$ of type $(1,1)$, a vector field $\xi$, a 1-form $\eta$ 
and a Riemannian metric $g$ such 
that
\begin{equation}\label{eq:1}
J^2 = -I + \eta \otimes \xi,\quad \eta(\xi)=1, \quad g(JX,JY)=g(X,Y)-\eta(X)\eta(Y),
\end{equation}
for vector fields $X$ and $Y$, $I$ the identity transformation on $TM$.
\end{definition}

A local $J$-basis for $TM$, $\{X_1,\ldots,X_n,JX_1,\ldots,JX_n,\xi\}$,
may be found with $\eta(X_i)=0$ for $i=1,\ldots,n$. The \emph{fundamental
$2$-form} on $M$ is given by
\[\omega(X,Y) = g(JX,Y),\]
and if $\{\alpha_1,\ldots,\alpha_n,\beta_1,\ldots,\beta_n,\eta\}$ is
a local $1$-form basis dual to the local $J$-basis, then
\[\omega = \sum_{i=1}^n \alpha_i \wedge \beta_i.\]
Note that $\imath_\xi \omega = 0$.

\begin{definition}
The geometric structure $(M^{2n+1},J,\xi,\eta,g)$ is 
\begin{itemize}
\item \textbf{co-symplectic} if $d\omega=0=d\eta$;
\item \textbf{normal} if $[J,J]+2\, d\eta \otimes \xi=0$;
\item \textbf{co-K\"ahler} if it is co-symplectic and normal; equivalently, if $J$ is parallel with respect to the metric $g$.
\end{itemize}
%
\end{definition}

Recently, co-symplectic geometry has attracted a great deal of interest, especially in the context of Poisson geometry, where co-symplectic structures are interpreted as corank 1 Poisson structures 
(see for instance \cite{CdNY1,CoMa,FMM,GMP,MT}). Sasakian structures also belong to this family; more precisely, they are normal structures such that $d\eta=\omega$ (see \cite{BoGa,CdNY2,CdNMY}).

Two crucial facts about co-K\"ahler manifolds are contained in the following lemma.
For a direct proof of these facts, see \cite{BO}.

\begin{lemma}\label{lem:parallel}
On a co-K\"ahler manifold, the vector field $\xi$ is Killing and parallel.
Furthermore, the 1-form $\eta$ is parallel and harmonic.
\end{lemma}

\noindent \lemref{lem:parallel} is a key point in \thmref{thm:cosympsplit}
below. In fact, in \cite{Li} it is shown that we can replace $\eta$ by a
harmonic integral form $\eta_\theta$ with dual parallel vector field
$\xi_\theta$ and associated metric $g_\theta$, $(1,1)$-tensor
$J_\theta$ and closed $2$-form $\omega_\theta$ with $i_{\xi_\theta}\omega_\theta
=0$. Then we have the following result of H. Li.

\begin{theorem}[\cite{Li}] \label{thm:maptor}
With the structure $(M^{2n+1},J_\theta,\xi_\theta,\eta_\theta,g_\theta)$,
there is a compact K\"ahler manifold $(K,h)$ and a Hermitian isometry
$\psi\colon K \to K$ such that $M$ is diffeomorphic to the mapping torus
\[K_\psi = \frac{K \times [0,1]}{(x,0)\sim (\psi(x),1)}\]
with associated fibre bundle $K \to M=K_\psi \to S^1$.
\end{theorem}

In \cite{BO}, the following refinement of Li's result is proved:

\begin{theorem}[\cite{BO}, Theorem 3.3]\label{thm:cosympsplit}
Let $(M^{2n+1},J,\xi,\eta,g)$ be a compact co-K\"ahler manifold with integral structure
and mapping torus bundle $K \to M \to S^1$. Then $M$ splits as
$M \cong S^1 \times_{\mathbb Z_m} K$,
where $S^1 \times K \to M$ is a finite cover with structure group
$\mathbb Z_m$ acting diagonally and by translations on the
first factor. Moreover, $M$ fibres over the circle $S^1/(\mathbb Z_m)$
with finite structure group.
\end{theorem}

The first true study of the topological properties of co-K\"ahler manifolds was
made in \cite{CdLM} where the focus was on things such as Betti numbers and
a modified Lefschetz property. The two results above allow us to say something
about the fundamental group and, moreover, to display the higher homotopy groups
as those of the constituent K\"ahler manifold $K$ (groups which, of course, are
generally unknown as well).

Here we use work of Verbitsky \cite{Verb} and the geometric structure of
co-K\"ahler manifolds to give a completely new decomposition of the cohomology
of a co-K\"ahler manifold in terms of the basic cohomology of the
associated transversally K\"ahler characteristic foliation. This leads to a
new, simpler proof of the ``Lefschetz'' property of \cite{CdLM}.


\section{Parallel forms and quasi-isomorphisms on co-K\"ahler manifolds}\label{sec:parallel}
In \cite{Verb}, Verbitsky shows that, in case a smooth Riemannian manifold has a parallel
form, one can define a derivation of the de Rham algebra whose kernel is quasi-isomorphic to
the manifold's real cohomology algebra. In this section we will use this construction in the context of
co-K\"ahler manifolds, where the 1-form $\eta$ is parallel.
Once again, we shall see that some topological properties of co-K\"ahler manifolds may be
derived from corresponding properties of K\"ahler manifolds. This can be interpreted as a
geometric incarnation of our hereditary principle.

Let $M$ be a smooth manifold and let $\Omega^*(M;\bR)$ be the (real) de Rham algebra. A linear map
$f\in\textrm{End}(\Omega^*(M;\bR))$ has degree $|f|$ if $f\colon \Omega^k(M;\bR)\to\Omega^{k+|f|}(M;\bR)$.
Every linear map $f\colon \Omega^1(M;\bR)\to\Omega^{|f|+1}(M;\bR)$ can be extended to a graded derivation
$\rho_f$ of $\Omega^*(M;\bR)$ by imposing the Leibniz rule, i.e.
\begin{eqnarray}\label{Leibnitz}
\rho_f\big|_{\Omega^0(M;\bR)} & = & 0\nonumber \\
\rho_f\big|_{\Omega^1(M;\bR)} & = & f\nonumber \\
\rho_f(\alpha\wedge\beta) & = & \rho_f(\alpha)\wedge\beta+(-1)^{|\alpha||f|}\alpha\wedge\rho_f(\beta).
\end{eqnarray}
where $\alpha,\beta\in \Omega^*(M;\bR)$ and $|\alpha|$ is the degree of $\alpha$. (While
this apparently well-known fact is used in \cite{Verb}, it is not proved there. See
\cite[Lemma 4.3]{FN} for a proof.) Given two linear operators
$f,\tilde f\in\textrm{End}(\Omega^*(M;\bR))$, their \textit{supercommutator}
is defined as
\begin{equation*}
\{f,\tilde f\}=f\circ \tilde f-(-1)^{|f||\tilde f|}\tilde f\circ f.
\end{equation*}

Let $(M,g)$ be a smooth Riemannian manifold and let $\eta\in\Omega^k(M;\bR)$ be a $k$-form. Define a linear map $\bar{\eta}\colon \Omega^1(M;\bR)\to \Omega^{k-1}(M;\bR)$, with $|\bar{\eta}|=k-2$, by
\[
\bar{\eta}(\nu)=\ii_{\nu^\#}\eta\,,
\]
where $^\# \colon T^*M\to TM$ is the isomorphism given by the metric. Denote by
$\rho_\eta\colon \Omega^{*}(M;\bR)\to \Omega^{*+k-2}(M;\bR)$ the corresponding derivation.
Define the linear operator $d_\eta\colon \Omega^{*}(M;\bR)\to
\Omega^{*+k-1}(M;\bR)$ as
\[
d_\eta=\{d,\rho_\eta\}\,.
\]
Since $d_\eta$ is the supercommutator of two graded derivations, one sees easily that it is
itself a graded derivation of degree $k-1$ and that it supercommutes with $d$. As a consequence, $\ker(d_\eta)\subset\Omega^*(M;\bR)$ is a differential subalgebra and has the structure of a cdga.
In \cite{Verb}, Verbitsky proves following:

\begin{theorem}\label{Verbitsky}
Let $(M,g,\eta)$ be a compact Riemannian manifold equipped with a parallel form $\eta$.
Then the natural embedding
\[
(\ker(d_\eta),d)\hookrightarrow(\Omega^*(M;\bR),d) 
\]
is a quasi-isomorphism.
\end{theorem}

Let $(M,g,\eta)$ be a Riemannian manifold equipped with a parallel form $\eta$.
\thmref{Verbitsky} says that we can recover the cohomology of $M$ by considering the subalgebra
of forms $\nu$ which are annihilated by $d_\eta$, i.e. those for which $d_\eta(\nu)=0$.
This allows one to greatly simplify, in many cases, the computation of the de Rham cohomology
of this kind of manifold.

Recall from \lemref{lem:parallel} that the 1-form $\eta$ is parallel on a co-K\"ahler manifold.
According to Verbitsky's construction, there is a derivation
$d_\eta$ of $(\Omega^*(M;\bR),d)$ described explicitly as follows.
\begin{lemma}
Let $(M,J,\eta,\xi,g)$ be a co-K\"ahler manifold. Then $d_\eta=L_\xi$, where $L_\xi$
denotes the Lie derivative in the direction of the vector field $\xi$.
\end{lemma}
\begin{proof}
Denote by $\bar{\eta}\colon\Omega^*(M;\bR)\to\Omega^*(M;\bR)$ the operator which acts on
1-forms as $\bar{\eta}(\nu)=\ii_{\nu^\#}\eta$. Since $|\bar{\eta}|=-1$, we have $d_\eta=\{d,\rho_\eta\}=d\circ\rho_\eta+\rho_\eta\circ d$, and $|d_\eta|=0$. To prove the lemma,
by \cite{FN}, it is enough to consider the action of $d_\eta$ on $0$- and $1$-forms.
Now, according to the formulas in (\ref{Leibnitz}) extending $\bar\eta$ to a derivation $\rho_\eta$,
on 1-forms we have $\rho_\eta=\bar\eta$ and
\[
\bar{\eta}(\nu)=\ii_{\nu^\#}\eta=\eta(\nu^\#)=g(\xi,\nu^\#)=\nu(\xi)=\ii_\xi\nu.
\]
Note that this identifies $\bar\eta=\ii_\xi$ which is already a derivation, so
$\rho_\eta=\ii_\xi$. Hence, $(d\circ\bar{\eta})(\nu)=d\ii_\xi\nu$ and, on the other hand,
$(\bar{\eta}\circ d)(\nu)= \ii_\xi(d\nu)$. By Cartan's magic formula, we obtain
\[
d_\eta(\nu)=(d\circ\bar{\eta})(\nu)+(\bar{\eta}\circ d)(\nu)=d\ii_\xi\nu+\ii_\xi(d\nu)=L_\xi(\nu).
\]
Thus $d_\eta=L_\xi$ on 1-forms. On a $0$-form (i.e. a function) $f$, we have
\[
d_\eta(f)=\rho_\eta(df)=\bar\eta(df)=df(\xi)=\xi(f)=L_\xi(f)
\]
by the calculation above. Since $d_\eta$ and $L_\xi$ are graded derivations of
the de Rham algebra which agree on $0$-forms and $1$-forms, the result follows.
\end{proof}

Let us consider the following graded differential subalgebra $(\Omega_\eta^*(M),d)$
of $(\Omega^*(M;\bR),d)$ given by
\[
\Omega_\eta^*(M)=\{\nu\in\Omega^*(M;\bR) \ | \ L_{\xi}(\nu)=0\}\,.
\]
As a consequence of \thmref{Verbitsky}, we obtain the following result.
\begin{corollary}\label{quism2}
On a compact co-K\"ahler manifold, the natural inclusion
\[
(\Omega^*_\eta(M),d)\hookrightarrow (\Omega^*(M;\bR),d)
\]
is a cdga quasi-isomorphism and
\[H^*(M;\R) \cong H^*_\eta(M),\]
where $H^*_\eta(M)$ is the cohomology of $(\Omega^*_\eta(M),d)$.
\end{corollary}

We shall use the cdga $\Omega^*_\eta(M)$ to give an alternative proof of the Lefschetz property
and of formality for co-K\"ahler manifolds in the hereditary framework of the rest of the paper.

Let $(M,J,\xi,\eta,g)$ be a compact co-K\"ahler manifold. In \cite{CdLM}, the authors defined a Lefschetz
map on harmonic forms and proved that it is an isomorphism. This is, of course, different from the K\"ahler context,
where the Lefschetz map can be defined directly on all forms and depends only on the underlying symplectic
structure, not on the metric. On forms, the Lefschetz map is $\cL^{n-p}\colon\Omega^p(M;\R)\to \Omega^{2n+1-p}(M,\R)$, given by
\begin{equation}\label{Lefschetz}
\a\mapsto\omega^{n-p+1}\wedge\ii_\xi\a+\omega^{n-p}\wedge\eta\wedge\a
\end{equation}
One sees immediately that the Lefschetz map does not send closed (resp. exact) forms to closed (resp. exact) forms, as it happens in the K\"ahler case,
hence does not descend to a map on cohomology. However, by restricting the Lefschetz map to the cdga
$\Omega^*_\eta(M)$, we \emph{are able to} descend to cohomology.
\begin{proposition}
The Lefschetz map \eqref{Lefschetz} restricts to a map
\[
\cL^{n-p}\colon\Omega^p_\eta(M)\to \Omega^{2n+1-p}_\eta(M) 
\]
for $0\leq p \leq n$, which sends closed (resp. exact) forms to closed (resp. exact) forms. Hence, $\cL$
descends to the cohomology $H^*_\eta(M)\cong H^*(M;\R)$.
\end{proposition}
\begin{proof}
We first show that if $\alpha\in\Omega^p_\eta(M)$, then $\cL^{n-p}(\alpha)\in\Omega^{2n+1-p}_\eta(M)$.
\begin{align*}
L_\xi(\cL^{n-p}(\alpha))&=L_\xi(\omega^{n-p+1}\wedge\ii_\xi \a+\omega^{n-p}\wedge\eta\wedge \a)=\omega^{n-p+1}\wedge L_\xi(\ii_\xi\a)=\\
&=\omega^{n-p+1}\wedge \ii_\xi d\imath_\xi\a=-\omega^{n-p+1}\wedge \ii_\xi\ii_\xi d\a=0,
\end{align*}
where we have used the facts that the Lie derivative $L_\xi$ is a derivation, $L_\xi=\imath_\xi d+ d \imath_\xi$ (Cartan's
Magic formula), $\ii_\xi \ii_\xi=0$ and $L_\xi\omega=L_\xi\eta=L_\xi\a=0$. For $\alpha$ a closed form in $\Omega^p_\eta(M)$, we have
\[
d(\cL^{n-p}(\alpha))=d(\omega^{n-p+1}\wedge\ii_\xi \a+\omega^{n-p}\wedge\eta\wedge \a)=\omega^{n-p+1}\wedge d \ii_\xi\a=0\,;
\]
for $\beta\in\Omega^{p-1}_\eta(M)$,
\begin{align*}
\cL^{n-p}(d\beta)&=\omega^{n-p+1}\wedge\ii_\xi d\beta+\omega^{n-p}\wedge\eta\wedge d\beta=\\
&=-\omega^{n-p+1}\wedge d \ii_\xi\beta-d(\omega^{n-p}\wedge\eta\wedge \beta)=\\
&=d(-\omega^{n-p+1}\wedge \ii_\xi\beta-\omega^{n-p}\wedge\eta\wedge \beta).
\end{align*}
\end{proof}

\noindent Consider the following two subalgebras of $\Omega^*_\eta(M)$:
\[
\Omega^p_1(M)=\{\alpha\in\Omega^p_\eta(M) \ | \ \imath_\xi\alpha=0\}, \
\Omega^p_2(M)=\bQ \oplus \{\alpha\in\Omega^p_\eta(M) \ | \ \eta\wedge\alpha=0\}.
\]
\begin{lemma}
$\Omega^p_\eta(M)=\Omega^p_1(M)\oplus\Omega^p_2(M)$ for all $p>0$ and $\Omega^*_i(M)$ is a differential
subalgebra of $\Omega^*_\eta(M)$, $i=1,2$.
\end{lemma}
\begin{proof}
Given any $\alpha\in\Omega^p_\eta(M)$, we can write tautologically
\begin{equation}\label{decomposition}
\a=(\a-\eta\wedge\ii_\xi\a)+\eta\wedge\ii_\xi\a\eqcolon\a_1+\a_2.
\end{equation}
Since $\eta(\xi)=1$, we see immediately that $\ii_\xi\a_1=0$, so $\a_1\in\Omega^p_1(M)$.
Clearly $\a_2\in\Omega^p_2(M)$. Now suppose that $\a\in\Omega^p_1(M)\cap\Omega^p_2(M)$.
Then $\eta\wedge\a=0$ and hence, by applying $\ii_\xi$, we get $0=\a-\eta\wedge\ii_\xi\a=\alpha$,
which gives $\alpha=0$.

Now, if $\a\in\Omega^p_\eta(M)$, then $L_\xi\a=d\ii_\xi\a+\ii_\xi d\a=0$, so $\ii_\xi d\a=-d\ii_\xi\a$.
If $\a\in\Omega^p_1(M)$, then we also have $\ii_\xi(d\a)=-d\ii_\xi\a=0$ since $\a\in\Omega^p_1(M)$.
Hence $d\colon \Omega^p_1(M)\to\Omega^{p+1}_1(M)$.

Finally, suppose $\a\in\Omega^p_2(M)$. Then, since $\eta$ is closed, we have $\eta\wedge d\a=-d(\eta\wedge\a)=0$.
Hence $d\colon \Omega^p_2(M)\to\Omega^{p+1}_2(M)$.
\end{proof}

As a consequence, the cohomology $H^p_\eta(M)$ of the cdga $\Omega^*_\eta(M)$ can be written as
\[
H^p_\eta(M)\cong H^p_1(M)\oplus H^p_2(M)\,,
\]
where $H^p_i(M)=H^p(\Omega^*_i(M))$, $i=1,2$. Now consider a form $\a\in\Omega^p_2(M)$. Applying the
derivation $\ii_\xi$ to the equation $\eta\wedge\a=0$, we obtain $\a=\eta\wedge\ii_\xi\a$, where clearly $\ii_\xi\a\in\Omega^{p-1}_1(M)$. This tells us that $\Omega^p_2(M)=\eta\wedge\Omega^{p-1}_1(M)$ and,
since $d\eta=0$, we have a differential splitting
\[
\Omega^p_\eta(M)=\Omega^p_1(M)\oplus \eta \wedge\Omega^{p-1}_1(M)\,.
\]
From this, we immediately deduce

\begin{corollary}\label{corollary1}
The cohomology $H^p_\eta(M)$ of $\Omega^*_\eta(M)$ splits as
\[
H^p_\eta(M)=H^p_1(M)\oplus [\eta]\wedge H^{p-1}_1(M)
\]
\end{corollary}
\noindent This corollary shows that the cohomology of $\Omega^*_\eta(M)$ only depends on the cohomology
of the cdga $\Omega^*_1(M)$.

Let us now consider the \emph{characteristic foliation} $\cF_\xi$ on a compact co-K\"ahler manifold
$(M,J,\xi,\eta,g)$ given by $(\cF_\xi)_x=\langle \xi_x\rangle$ for every $x\in M$. Such a foliation
is Riemannian and transversally K\"ahler. Indeed, at every point $x\in M$, the orthogonal space to
$\xi$ is endowed with a K\"ahler structure given by $(J,g,\omega)$, and all these data vary smoothly
with $x$.

Recall that, given a foliation $\cF$ on a compact manifold $M$, the \emph{basic cohomology} is defined as
the cohomology of the complex $\Omega^*(M,\cF)$, where
\[
\Omega^p(M,\cF)=\{\a\in\Omega^p(M) \ | \ \imath_X\a=\imath_Xd\a=0 \  \forall X\in\mathfrak{X}(\cF)\}
\]
and $\mathfrak{X}(\cF)$ denotes the subalgebra of vector fields tangent to $\cF$.
In our case, we have the following.
\begin{lemma}\label{Lemma1}
Let $(M,J,\xi,\eta,g)$ be a compact co-K\"ahler manifold and let $\cF_\xi$ be the characteristic foliation.
Then $\Omega^*_1(M)=\Omega^*(M,\cF_\xi)$.
\end{lemma}
\begin{proof}
This is clear, since
\[
\a\in \Omega^*_1(M)\Leftrightarrow L_\xi\a=\ii_\xi\a=0\Leftrightarrow\imath_\xi d\a=\ii_\xi\a=0\Leftrightarrow \a\in \Omega^p(M,\cF_\xi).
\]
\end{proof}

\begin{corollary}\label{corollary2}
On a compact co-K\"ahler manifold $M$, $H^*_1(M)\cong H^*(M,\cF_\xi)$ and
\[
H^*(M;\R) \cong H^*_\eta(M)=H^*(M,\cF_\xi)\oplus [\eta]\wedge H^{*-1}(M,\cF_\xi)\,.
\]
\end{corollary}
\begin{theorem}
Let $(M,J,\xi,\eta,g)$ be a compact co-K\"ahler manifold. Then the Lefschetz map
\begin{align*}
\cL^{n-p}\colon H^p(M;\R) & \cong H^p_\eta(M) \to H^{2n+1-p}_\eta(M)\cong H^{2n+1-p}(M;\R), \\
\a & \mapsto\omega^{n-p+1}\wedge\ii_\xi\a+\omega^{n-p}\wedge\eta\wedge\a
\end{align*}
is an isomorphism for $0\leq p\leq n$.
\end{theorem}
\begin{proof}
First note that, by Poincar\'e duality, it is sufficient to show that $\cL^{n-p}$ has
zero kernel. Now, by Corollary \ref{quism2}, on a compact co-K\"ahler manifold we have an isomorphism
$H^p_\eta(M)\cong \cH^p(M)$. In particular, Corollary \ref{corollary2} tells us that the (harmonic)
cohomology of $M$ can be computed as a cylinder on the basic cohomology of the characteristic foliation.
Since the latter is transversally K\"ahler, in view of \cite{EKA}, the map
$H^p(M,\cF_\xi)\to H^{2n-p}(M,\cF_\xi)$ given by multiplication with the K\"ahler form
$\omega^{n-p}$ is an isomorphism for $p\leq n$. Again by Corollary \ref{corollary2}, the corresponding
map $H^p_1(M)\to H^{2n-p}_1(M)$ is also an isomorphism.

Now consider the Lefschetz map $\cL^{n-p}\colon H^p_\eta(M)\to H^{2n+1-p}_\eta(M)$
given by (\ref{Lefschetz}). Decompose any $\a \in H^p_\eta(M)$ as $\a=\a_1+\a_2$ according to
(\ref{decomposition}) so that $\ii_\xi \alpha_1=0$ and $\a_2=\eta\wedge\ii_\xi\a$. We shall show that
the Lefschetz map is non-zero on both $\a_1$ and $\a_2$ with $\cL^{n-p}(\a_1) \in
\eta \wedge H^{2n-p}_1(M)$ and $\cL^{n-p}(\a_2) \in H^{2n+1-p}_1(M)$. Then, because these
sub-algebras are complementary, we will have $\cL^{n-p}(\a) \neq = 0$ for all $\a \neq = 0$.

For $\a_1\in H^p_1(M)\cong H^p(M,\cF_\xi)$, because $\ii_\xi\a_1=0$, the first term in the
Lefschetz map definition applied to $\a_1$ vanishes. Hence, we get that
$\omega^{n-p}\wedge\a_1\neq 0$ in $H^{2n-p}(M,\cF_\xi)$ and, in view of Corollary \ref{corollary2},
this implies that $\omega^{n-p}\wedge\eta\wedge\a_1$ is non-zero in
$\eta \wedge H^{2n-p}_1(M) \subseteq H^{2n+1-p}_\eta(M)$.

Because $\a_2=\eta\wedge\ii_\xi\a$, we see that the second term in the Lefschetz map definition
applied to $\a_2$ vanishes. Now, $\ii_\xi \a_2\in H^{p-1}_1(M)\cong H^{p-1}(M,\cF_\xi)$, so
$\omega^{n-p+1}\wedge\ii_\xi\a_2\neq 0$ in $H^{2n-p+1}(M,\cF_\xi)\cong H^{2n-p+1}_1(M)$.
Therefore, when $p\geq 1$,
\begin{align*}
\cL^{n-p}(\alpha) & =\omega^{n-p+1}\wedge\ii_\xi\a+\omega^{n-p}\wedge \eta \wedge\a \\
& = \omega ^{n-p+1}\wedge\ii_\xi\a_2+ \omega ^{n-p}\wedge \eta \wedge\a_1 \\
& \neq 0,
\end{align*}
so $\cL^{n-p}$ has zero kernel and is thus an isomorphism on cohomology. Furthermore, when $p=0$, we get
\[
\cL^n(1)=\omega^n\wedge \eta\neq 0\,,
\]
since $\omega^n\wedge\eta$ is a volume form by assumption and, hence, cannot be exact.
\end{proof}
\noindent Since $H^p_\eta(M)\cong\cH^p(M)$ (harmonic forms) on a compact co-K\"ahler manifold, we obtain
\begin{corollary}
Let $(M,J,\xi,\eta,g)$ be a compact co-K\"ahler manifold. Then the Lefschetz map
$\cL^{n-p}\colon\cH^p(M)\to \cH^{2n+1-p}(M)$ 
is an isomorphism for $0\leq p\leq n$.
\end{corollary}

In \cite{CW} the authors prove that the minimal model $\cM_{M,\cF}$ of the basic forms
$\Omega^*(M,\cF)$ of a transversally
K\"ahler foliation $\cF$ on a compact manifold is formal. We would like to use our characterization
(in a slightly different form) of the cohomology of a compact co-K\"ahler manifold to give an
alternative proof of this formality in the context of co-K\"ahler geometry as well as a new description of the minimal model of a co-K\"ahler
manifold. Note that \corref{corollary2} may be phrased as the following.

\begin{corollary}\label{cor:tensorcor}
Let $M$ be a compact co-K\"ahler manifold; then $H^*_1(M)\cong H^*(M,\cF_\xi)$ and
\[
H^*(M;\R) \cong H^*_\eta(M)=H^*(M,\cF_\xi)\otimes \land([\eta])\,.
\]
Furthermore, the splitting $\Omega^p_\eta(M)=\Omega^p_1(M)\oplus \eta \wedge\Omega^{p-1}_1(M)$
(for each $p$) may be written as
\[
\Omega^*_\eta(M)=\Omega^*_1(M)\otimes \land(\eta)\,. 
\]
\end{corollary}

\noindent Using this description, we can now see the transversally K\"ahler structure
reflected in the minimal model of $M$. We

\begin{proposition}
Let $M$ be a compact co-K\"ahler manifold. Then $\cM_{M,\cF}$ is formal in the sense of Sullivan
and the minimal model of $M$ splits as a tensor product of cdga's
\[
\cM_M \cong \cM_{M,\cF} \otimes \land(\eta,d=0).
\]
\end{proposition}
\begin{proof}
We use two facts: first, by \cite{CdLM,BLO}, we know that $M$ is formal; secondly, we know that, in a cdga decomposition
$A\cong B \otimes C$, $A$ is formal if and only if both $B$ and $C$ are formal. 
\end{proof}

\begin{remark} The proof above is much simpler than the original in \cite{CW}, but is only for transversally K\"ahler
foliations arising from co-K\"ahler structures. Of course, if we, on the other hand, assume 
the formality of $\cM_{M,\cF}$ (by \cite{CW}), then $\land(\eta,d=0)$ and the identification
$\Omega^*(M,\cF) \cong \Omega^*_1(M)$ allow us to obtain the following diagram.
\[\xymatrix{
H^*_\eta(M) \ar[r]^-\cong & H^*(M,\cF_\xi)\otimes \land([\eta]) \\
\cM_M \ar@{-->}[u]^-\theta\ar@{-->}[r]^-\rho \ar[d]^\simeq & \cM_{M,\cF} \otimes \land(\eta,d=0)
\ar[d]_\simeq \ar[u]_\simeq \\
\Omega^*_\eta(M) \ar@{=}[r] & \Omega^*(M,\cF)\otimes \land(\eta)
}
\]
Here, the quasi-isomorphism $\rho$ is obtained from a standard lifting lemma in minimal model theory applied to the bottom part of the
diagram. By composition, we then obtain $\theta$ and we see it is a quasi-isomorphism. Hence,
$M$ is formal and, again by the lifting lemma, the quasi-isomorphism $\rho$ is an
isomorphism.
\end{remark}

\end{document}